\documentclass[a4paper,11pt,openany]{amsart}
\usepackage[T1]{fontenc}
\usepackage[utf8]{inputenc}
\usepackage {amsmath , amssymb, mathrsfs}
\usepackage {setspace}
\usepackage {amsthm}
\usepackage{graphicx, color}
\usepackage{tabulary}
\newtheorem{cor}{Corollary}[section]

\newtheorem {teo}[cor]{Theorem}
\newtheorem {prop}[cor]{Proposition}
\newtheorem {lemm}[cor]{Lemma}

\newtheorem{ex}[cor]{Example}
\newtheorem{oss}[cor]{Remark}

\newcommand{\hada} {\star}

\newcommand{\PP} {{\mathbb{P}}}
\newcommand\undermat[2]{
	\makebox[0pt][l]{$\smash{\underbrace{\phantom{%
					\begin{matrix}#2\end{matrix}}}_{\text{$#1$}}}$}#2}

\usepackage[all]{xy}
\begin{document}
	\title {Hadamard products of degenerate subvarieties }

	\author{G. Calussi}
	\address{Dipartimento di Matematica e Informatica "Ulisse Dini", Universit\`a di Firenze \\
		Viale Morgagni 67/a, 50134 Firenze, Italy
		\vskip0.1cm Dipartimento di Matematica e Informatica, Universit\`a di Perugia \\
		via Vanvitelli 1, 06123 Perugia, Italy}
	\email{gabriele.calussi@gmail.com}
	
	\author{E. Carlini}
	\address{DISMA-Dipartimento di Scienze Matematiche, Politecnico di Torino \\
		Corso Duca degli Abruzzi 24, 10129 Torino, Italy}
	\email{enrico.carlini@polito.it}
	
	\author{G. Fatabbi}
	\address{Dipartimento di Matematica e Informatica, Universit\`a di Perugia \\
		via Vanvitelli 1, 06123 Perugia, Italy}
	\email{giuliana.fatabbi@unipg.it}
	\author{A. Lorenzini}
	\address{Dipartimento di Matematica e Informatica, Universit\`a di Perugia \\
		via Vanvitelli 1, 06123 Perugia, Italy}
	\email{annalor@dmi.unipg.it}
	\date{\today}
	\thanks{The first, the third and the fourth author thank the Politecnico of Torino for its support while visiting the second author.
		\\The first author thanks GNSAGA of INdAM and MIUR for their partial support.}
	
\begin{abstract}
We consider generic degenerate subvarieties $X_i\subset\mathbb{P}^n$. We determine an integer $N$, depending on the varieties, and for $n\geq N$ we compute dimension and degree formulas for the Hadamard product of the varieties $X_i$. Moreover, if the varieties $X_i$ are smooth, their Hadamard product is smooth too. For $n<N$, if the $X_i$ are generically $d_i$-parameterized, the dimension and degree formulas still hold. However, the Hadamard product can be singular and we give a lower bound for the dimension of the singular locus.
\end{abstract}

\maketitle

\section{Introduction}
The Hadamard product of matrices is a well established operation in Mathematics having several connections both theoretical and applied, for example see the recent entry in T. Tao's blog about the paper \cite{KhareTao2017}.
\par  More recently the definition of the {\em Hadamard product} between subvarieties $X,Y$ of projective space, denoted $X \hada Y$, has been introduced by \cite{CMS} as the  closure of the image of the rational map
\[
X \times Y \dashrightarrow  {\mathbb{P}}^n, \quad  ([a_0:\dots : a_n], [b_0: \dots :b_n])\mapsto [a_0b_0 :a_1b_1:\ldots :a_nb_n].
\]

This product is far less studied and our knowledge is still at a developing stage but it is attracting quite a lot of attention and applications have been shown, for example, to Algebraic Statistics (see \cite{CMS,CTY}).
In particular, in \cite{CMS} it is shown that the restricted Boltzmann machine is a graphical model for binary random variables, starting with the observation that its Zariski closure is a Hadamard power of the first secant variety of the Segre variety of projective lines.
\par
The Hadamard product of varieties is also related to tropical geometry, for instance the tropicalization of the Hadamard product of two varieties is the Minkowski sum of the tropicalizations of the two varieties (see \cite[Proposition 5.1]{BCK}, \cite{FOW}, \cite{MS}).

One of the most important open question is to find the dimension and the degree of the Hadamard product of varieties. In \cite{BCK} the authors give formulas for the dimension and the degree of the Hadamard product of general linear spaces. They also introduce the notion of expected dimension of the Hadamard product of irreducible varieties. In \cite{FOW} the authors give an expected formula for the degree of the Hadamard product of varieties in general position. In this paper we prove that the expected formula holds for the Hadamard product of generic degenerate subvarieties if the ambient space is large enough, thus partially answering \cite[Question 1.1]{FOW}.

\par We work over an algebraically closed  field $\mathbb{K}$ of characteristic $0$ and we assume that all the varieties we consider are irreducible.
\par Given a subvariety $V$ of projective space we denote by $I_V$ its (saturated radical homogeneous) ideal and by $HF_V$ its {\em Hilbert function}; that is, $HF_V(t)=dim_\mathbb{K}R_t/(I_V)_t$ where $R=\mathbb{K}[x_0,\dots,x_n]=\bigoplus\limits_{t\geq 0} R_t$, $R_t$ is the vector space of the homogeneous polynomials of degree $t$ and $(I_V)_t=I_V\cap R_t$.

\par In Section \ref{secprod} we prove that, if $X_1,\dots, X_\ell$ are $\ell$ degenerate subvarieties of $\PP^n$ whose linear spans are generic of dimension $h_1,\dots,h_\ell$ respectively, with $n\geq (h_1+1)\cdots(h_\ell+1)-1$, then the Hadamard product $X_1 \hada\dots\hada X_\ell$ and the product variety $X_1\times\dots\times X_\ell$ are projectively equivalent as subvarieties of $\PP^n$. As a consequence we obtain that the dimension of $X_1 \hada\dots\hada X_\ell$ is the sum of the dimensions, the degree is the product of the degrees multiplied by a multi-binomial coefficient depending on the dimensions and the Hilbert function is the product of the Hilbert functions. These degree and dimension formulas generalize the ones in \cite[Theorem 6.8]{BCK} which are only given for linear spaces. We also prove that, if the varieties $X_i$ are smooth, then their Hadamard product is smooth.

\par In Section 3 we consider two generically $d_X$-parameterized and $d_Y$-parameterized subvarieties of $\PP^n$ of dimension $r$, $s$, respectively, with $N-\left(\binom{r+d_X}{d_X}+\binom{s+d_Y}{d_Y}-2\right)\leq n\leq N-1$ where $N=\binom{r+d_X}{d_X}\binom{s+d_Y}{d_Y}-1$. In this case the formula for the Hilbert function no longer holds, but we still have the dimension and degree formulas. We also extend these results to a finite number of subvarieties. In this situation singularities may arise even if the varieties are smooth: on one hand we give a numerical condition sufficient for smoothness and, on the other hand, we give a numerical condition sufficient for the Hadamard product to be singular. In the latter case,  we give a lower bound for the dimension of the singular locus.
\par We conclude with some explicit examples in Section \ref{section4}. These examples show the role of the genericity assumption and how singularities can arise.
\vskip.2cm
\par We wish to thank B. Sturmfels for some useful conversations and suggestions.
\par We wish to thank the Referees for their careful reading of the paper and the helpful comments and suggestions.
\section{Large ambient space}\label{secprod}

In this section we consider the Hadamard product of subvarieties whose linear spans are generic and in particular the case in which the ambient space has dimension large enough in a very precise sense. Note that \cite[Theorem 4.1]{BCFL2} considered the product of generic linear spaces.

In what follows we embed $\PP^{h_1}\times\cdots \times \PP^{h_\ell}$ in $\PP^n$, where $n\geq N=(h_1+1)\cdots(h_\ell+1)-1$, via the composition of the usual Segre embedding of $\PP^{h_1}\times\cdots \times \PP^{h_\ell}$ in $\PP^N$ with the immersion $\PP^N \hookrightarrow \PP^n$ given by $[a_0:\dots :a_N]\mapsto [a_0: \dots :a_N:0:\dots:0]$. We still call this composition Segre embedding and we denote it by $\sigma$.

We also need to recall the {\em Khatri-Rao product} (developed by single rows) of two matrices (see \cite{KR}): given a $p\times q$-matrix $A=\left(a_{ij}\right)$ and a $p\times t$ matrix $B=\begin{pmatrix}B_1 \\ \vdots \\ B_p \end{pmatrix}$, the Khatri-Rao product (developed by single rows) is the $p\times qt$-matrix:
$$
A\otimes_{KR} B=\begin{pmatrix}a_{11}B_1 & \cdots & a_{1q}B_1 \\ \vdots & \vdots& \vdots \\ a_{p1}B_p & \cdots & a_{pq}B_p \end{pmatrix}.
$$
\par We first state a technical Lemma on the Khatri-Rao product of special matrices which will be used in the proof of Theorem \ref{teoX1timesXkbign2}. The proof of this Lemma is a straightforward computation.
\par If $s$ and $k$ are positive integers, we denote by $N^{^s}_{_k}$ the matrix of size $ks\times s$
\[
N^{^s}_{_k}=\left(\begin{array}{c}
\rotatebox[origin=c]{90}{\text{{\scriptsize   {\it k} -times}}} \left\{\begin{array}{cccccc}
  1 & 0 & 0 & \dots & 0 & 0\\
  \vdots & \vdots &\vdots & \cdots & \vdots & \vdots \\
  1 & 0 &  0 &\dots & 0 & 0
\end{array}\right.\\
\rotatebox[origin=c]{90}{\text{{\scriptsize   {\it k} -times}}} \left\{  \begin{array}{cccccc}
  0 & 1 & 0 & \dots & 0 & 0\\
  \vdots & \vdots &\vdots & \cdots & \vdots & \vdots \\
  0 & 1 &  0 &\dots & 0 & 0
\end{array}\right.\\
\vdots

\\
\rotatebox[origin=c]{90}{\text{{\scriptsize   {\it k} -times}}} \left\{ \begin{array}{cccccc}
  0 & 0 & 0 & \dots & 0 & 1\\
  \vdots & \vdots &\vdots & \cdots & \vdots & \vdots \\
  \undermat{s}{ 0 & 0 &  0 &\dots & 0 & 1}
\end{array}\right.\\

\end{array}\right).
\]
\vskip 1cm
Note that, if $k=1$, then $N^{^s}_{_1}=I_s$, where $I_s$ denotes the identity matrix of size $s$.

\vskip 1cm
\begin{lemm}\label{lemmaKR}
Let $a,b,c,n$ be positive integers with $n\geq abc$. If $
A$ is the $n\times a$ matrix $\begin{pmatrix} N^{^a}_{_{bc}}\\
\\
\text{\LARGE 0}
\end{pmatrix}
$ and
$B$ is the $n\times b$ matrix $\left(\begin{array}{c}
\rotatebox[origin=c]{90}{\text{{\scriptsize   {\it a} -times}}} \left\{\begin{array}{c}
 N^{^b}_{_{c}}\\ \vdots \\N^{^b}_{_c}
\end{array}\right.\\
\\
\ \ \ \text{ \LARGE 0}
\end{array}
\right)
$, then
$
A\otimes_{KR}B$ is the $n\times ab$ matrix $\begin{pmatrix} N^{^{ab}}_{_c}\\
\\
\text{\LARGE 0}
\end{pmatrix}
$.
\end{lemm}

\begin{teo}\label{teoX1timesXkbign2}
Let $\ell$ be a positive integer. Let $X_1, \dots , X_\ell$ be subvarieties of $\PP^n$ whose linear spans are generic of dimension $h_1,\dots,h_\ell$, respectively. If $n\geq N=(h_1+1)\cdots(h_\ell+1)-1$, then the Hadamard product $X_1 \hada \cdots \hada X_\ell$ and the product variety $X_1\times \cdots \times X_\ell$ are projectively equivalent as subvarieties of $\PP^n$.
\end{teo}
\begin{proof}
Let $L_1,\dots,L_\ell$ be the linear spans of the subvarieties $X_1, \dots , X_\ell$. Assume that $L_i$  have parametric equations given respectively  by
\[
L_i:\begin{cases}
x_0=f_{i0}(y_{i0}, \dots, y_{ih_i})\\
x_1=f_{i1}(y_{i0}, \dots, y_{ih_i})\\
\qquad  \qquad \vdots \\
x_n=f_{in}(y_{i0}, \dots, y_{ih_i})\\
\end{cases}
\]
where $f_{ij}(y_{i0}, \dots, y_{ih_i})=a_{j0}^{(i)}y_{i0}+a_{j1}^{(i)}y_{i1}+\cdots +a_{jh_i}^{(i)}y_{ih_i}$ for $i=0, \dots, \ell$ and for $j=0, \dots, n$ .
For each $i=0, \dots, \ell$ consider the matrix of size $(n+1)\times (h_i+1)$ defined as
 \[
M_i
=\begin{pmatrix}
a_{00}^{(i)} &   \dots &a_{0h_i}^{(i)}\\
& &\\
a_{10}^{(i)} &  \dots & a_{1h_i}^{(i)}\\
& &\\
\vdots & \vdots &  \vdots \\
& &\\
a_{n0}^{(i)} & \dots &  a_{nh_i}^{(i)}\\
\end{pmatrix}.
\]
Now consider the matrix $M'$ of size $(n+1)\times (N+1)$  given by the Khatri-Rao product (developed by single rows)
$$
M'=M_1\otimes_{KR}\cdots\otimes_{KR}M_\ell.
$$
Associating to each $f_{ij}$ a point of $(\PP^{h_i})^{^*}$, we can associate each $L_i$ to a point of $\undermat{n \mbox{ times}}{(\PP^{h_i})^{^*}\times\cdots\times (\PP^{h_i})^{^*}}$.
\par
\vskip.75cm

Now, the zero locus defined by the maximal minors of $M'$ which are multi-homogeneous polynomials, is a closed set whose complement is an open subset of $ \undermat{n \mbox{ times}}{(\PP^{h_1})^{^*}\times\cdots\times (\PP^{h_1})^{^*}}\times\cdots\times \undermat{n \mbox{ times}}{(\PP^{h_\ell})^{^*}\times\cdots\times (\PP^{h_\ell})^{^*}}$.
\vskip .75cm
\par
To see that such an open set is non-empty choose $$
M_i=\left(\begin{array}{c}
\rotatebox[origin=c]{90}{\text{{\footnotesize  {\it \underline{i}} -times}}} \left\{ \begin{array}{c}
N^{^{h_i+1}}_{_{\overline{i}}}\\
\vdots\\
N^{^{h_i+1}}_{_{\overline{i}}}\\
\end{array}\right.\\
\\
\ \ \ \text{\LARGE 0}
\end{array}
\right)
$$
where $\underline{i}=(h_{1}+1)\cdots(h_{i-1}+1)$  and $\overline{i}=(h_{i+1}+1)\cdots(h_{\ell}+1)$, with the understanding that $\underline{1}=1$ and $\overline{\ell}=1.$ By recursively using Lemma \ref{lemmaKR}, we obtain that $$((M_1\otimes_{KR} M_2)\otimes_{KR}\cdots )\otimes_{KR} M_\ell=\begin{pmatrix} N^{^{\underline{\ell+1}}}_{_{1}}\\
\\
\text{\LARGE 0}
\end{pmatrix}=\begin{pmatrix} I_{\underline{\ell+1}}\\
\\
\text{\LARGE 0}
\end{pmatrix}.
$$
Since $\underline{\ell+1}=N+1$, we have that $M'=\left(\begin{matrix} I_{N+1}\\ 0\end{matrix}\right)$.

\par Therefore the genericity of  $L_1,\dots,L_\ell$ gives that the matrix $M'$ has maximum rank $N+1$.
\par For $n>N$, we can complete the matrix $M'$ to a matrix $M$ of size $(n+1)\times (n+1)$ with $det(M)\not=0$, so that $M$ gives a projective isomorphism.
Let $\sigma$ be the Segre embedding of $\PP^{h_1}\times\cdots\times \PP^{h_\ell}$ in $\PP^n$, as defined at the beginning of this section. A direct computation shows that $$\left(M\circ\sigma\right)\left( [y_{10}:\dots : y_{1h_1}],\dots,[y_{\ell0}:\dots :z_{\ell k}]\right)=$$ $$[f_{10}(y_{10},\dots , y_{1h_1})\cdots f_{\ell 0}(y_{\ell 0},\dots ,y_{\ell h_\ell}):\ldots:f_{1n}(y_{10},\dots , y_{1h_1})\cdots f_{\ell n}(y_{\ell 0},\dots ,y_{\ell h_\ell})].$$
\par Therefore $M(P)\in \Sigma=\{P_1\hada\cdots\hada P_\ell | P_i\in L_i \}$ and so the map $\PP^n\stackrel{M}{\to} \PP^n$   sends each point $P\in \sigma(\PP^{h_1}\times\cdots\times \PP^{h_\ell})$ to a point $M(P)\in \Sigma \subseteq L_{1} \hada\cdots\hada L_\ell$.
Since $h_1+\cdots +h_\ell=dim (\sigma(\PP^{h_1}\times\cdots\times \PP^{h_\ell})\leq dim (L_{1} \hada\cdots\hada L_\ell)\leq h_1+\cdots +h_\ell,$ they are projectively equivalent.
Thus we have that $L_{1} \hada\cdots\hada L_\ell=\Sigma$. Therefore $P_1\hada\dots \hada P_\ell$ is always well-defined for any points $P_i\in L_i$ and thus for any points $P_i\in X_i$.
\par We just proved that $$M\left(\sigma\left(X_1\times \cdots \times X_\ell\right)\right)\subseteq\{P_1\hada\cdots \hada P_\ell|P_i\in X_i \text{ for all } i=0, \dots, \ell \}\subseteq X_1\hada \cdots \hada X_\ell.$$ Since $dim(X_1)+\cdots+dim(X_\ell)=dim (M(\sigma(X_1\times \cdots \times X_\ell)))\leq dim (X_1\hada \cdots \hada X_\ell)\leq dim(X_1)+\cdots +dim(X_\ell),$ they are projectively equivalent, as we wished.
\end{proof}
\begin{oss}\label{rmkM'}{\rm
Note that in the case of two subvarieties $X_1$ and $X_2$ the matrix $M'$ defined in the proof of the previous Theorem is given by:
\[
M'
=\begin{pmatrix}
a_{00}b_{00} &   \dots & a_{00}b_{0h_2} &a_{01}b_{00} & \dots & a_{0h_1}b_{0h_2}\\
a_{10}b_{10} &  \dots & a_{10}b_{1h_2} & a_{11}b_{10} & \dots & a_{1h_1}b_{1h_2}\\
\vdots & \vdots & \vdots & \vdots & \ddots & \vdots \\
a_{n0}b_{n0} & \dots & a_{n0}b_{nh_2} & a_{n1}b_{n0} &  \dots & a_{nh_1}b_{nh_2}\\
\end{pmatrix}
\]
if the linear spans of $X_1$  and $X_2$ have parametric equations given respectively  by
{\small \[
L_1:\begin{cases}
x_0=a_{00}y_0+a_{01}y_1+\cdots +a_{0h_1}y_{h_1}\\
x_1=a_{10}y_0+a_{11}y_1+\cdots +a_{1h_1}y_{h_1}\\
\qquad  \qquad \vdots \\
x_n=a_{n0}y_0+a_{n1}y_1+\cdots +a_{nh_1}y_{h_1}\\
\end{cases}
L_2:\begin{cases}
x_0=b_{00}z_0+b_{01}z_1+\cdots +b_{0k}z_{h_2}\\
x_1=b_{10}z_0+b_{11}z_1+\cdots +b_{1k}z_{h_2}\\
\qquad  \qquad \vdots \\
x_n=b_{n0}z_0+b_{n1}z_1+\cdots +b_{nk}z_{h_2}
\end{cases}.
\]}
In the proof of Theorem \ref{teoX1timesXkbign2} the genericity of the linear spans $L_1$ and $L_2$ is only used to say that the matrix $M'$ has maximum rank $N+1$, and so we can characterize a closed set $C$ of $$\left( \undermat{n+1 \mbox{ times}}{\PP^{h_1}\times\cdots\times \PP^{h_1}}\right)\times\left( \undermat{n+1 \mbox{ times}}{\PP^{h_2}\times\cdots\times \PP^{h_2}}\right)$$
\vskip.75cm
\noindent as the zero locus of the maximal minors of $M'$ which are multi-homogeneous polynomials of the multi-graded ring $$\mathbb{K}[a_{00},\dots,a_{0h_1},\dots, a_{n0},\dots,a_{nh_1},b_{00},\dots,b_{0h_2},\dots, b_{n0},\dots,b_{nh_2}].$$
\par The complement of $C$ is  an  open subset and each point of this open subset gives a parameterization of two linear subspaces of $\PP^n$ of dimensions $h_1$ and $h_2$ respectively. For subvarieties of these two linear subspaces Theorem \ref{teoX1timesXkbign2} holds. Moreover, in Theorem \ref{teoX1timesXkbign2} we proved that such an open set is non-empty.
}
\end{oss}

\begin{oss}{\rm
Note that Theorem \ref{teoX1timesXkbign2} generalizes \cite[Theorem 4.1]{BCFL2} in two directions:  we consider not only the product of linear spaces, but also the product of degenerate subvarieties, and we also consider an ambient space of larger dimension.
}\end{oss}
\begin{oss}\label{closure}{\rm
Under the assumptions of Theorem \ref{teoX1timesXkbign2}, in the proof of the Theorem, we also showed that $X_1\hada \cdots \hada X_\ell$ which is, by definition, $\overline{\{P_1\hada \cdots \hada P_\ell|P_i\in X_i \}}$ turns out to be $\{P_1\hada\cdots \hada P_\ell|P_i\in X_i\}.$
In particular the notation $P_1\hada \cdots \hada P_\ell\in \PP^n$ is well-defined for all $P_i\in X_i$ by the genericity assumptions.}
\end{oss}
\vskip 0.5cm
\begin{oss}\label{rmkNONsingolare}{\rm
An immediate consequence of Theorem \ref{teoX1timesXkbign2} is that if $X_1,\dots,X_\ell$ are non-singular, then also $X_1\hada\cdots\hada X_\ell$ is non-singular.}
\end{oss}
Theorem \ref{teoX1timesXkbign2} yields the following Corollary which extends the dimension and the degree formulas of  \cite[Theorem 6.8]{BCK} beyond linear spaces.
\begin{cor}\label{corX1timesXkbign}
Let $\ell$ be a positive integer. Let $X_1, \dots , X_\ell$ be subvarieties of $\PP^n$ of dimension $r_1,\dots,r_\ell$ whose linear spans are generic of dimension $h_1,\dots,h_\ell$, respectively. If $n\geq (h_1+1)\cdots(h_\ell+1)-1$, then
\begin{enumerate}
  \item[{\it i)}] $dim(X_1\hada\cdots\hada X_\ell)=\sum\limits_{i=1}^{\ell} dim(X_i)$
  \item[{\it ii)}] $deg(X_1\hada\cdots\hada X_\ell)=\binom{r_1+\cdots+r_\ell}{r_1,\dots,r_\ell}\prod\limits_{i=1}^{\ell} deg(X_i)$, where $\binom{r_1+\cdots+r_\ell}{r_1,\dots,r_\ell}=\frac{(r_1+\cdots+r_\ell)!}{r_1!\cdots r_\ell!}$
  \item[{\it iii)}] $HF_{X_1\hada\cdots\hada X_\ell}=\prod\limits_{i=1}^{\ell} HF_{X_i}.$
\end{enumerate}
\end{cor}

\begin{oss}{\rm
In Example \ref{Ex2.7}, we shall see an explicit example of two varieties $X$ and $Y$ whose linear spans are not generic and whose matrix $M'$, constructed as in the proof of Theorem \ref{teoX1timesXkbign2}, will not have maximum rank. Indeed, $X\hada Y$ is neither projectively equivalent nor isomorphic to the product variety $X\times Y$. In fact, in Example \ref{Ex2.7}, $\mbox{Sing}(X\hada Y)\neq \emptyset$, even if $X$ and $Y$ are smooth.
In that example, the dimension and degree formulas still hold, but the Hilbert function formula does not hold.}
\end{oss}

Before ending this section we introduce the notion of  generically $d$-parameterized subvariety, which allows us to extend our results to the case of a small ambient space.

\begin{oss}\label{rmkgennonsingular}{\rm
Given a subvariety $Z\subset\PP^n$ of dimension $r$  we say that it has a {\em parametric representation of degree} $d$ if $Z$ is the image of the rational map $\PP^r\dashrightarrow\PP^n$ defined by $n+1$ homogeneous polynomials of degree $d$ in $r+1$ indeterminates. It is clear that the linear span of $Z$ is of dimension at most $\binom{r+d}{d}-1$. Thus $Z$ is degenerate as soon as $\binom{r+d}{d}-1<n$.
We say that $Z$ is a {\it generically $d$-parameterized subvariety}, if the $n+1$ degree $d$ homogeneous polynomials in $r+1$ indeterminates  defining it are generic. Note that, if $Z$ is generically $d$-parameterized and $n\geq \binom{r+d}{d}-1,$ then the linear span of $Z$ is of dimension $\binom{r+d}{d}-1$ and $Z$ is non-singular, in fact it is projectively equivalent to the $d$-uple Veronese embedding of $\PP^r$.

}
\end{oss}
\begin{cor}\label{corteoXtimesYgen}
Let $\ell$ be a positive integer. For $i=1,\dots,\ell$, let $r_i, d_i$ be positive integers and let $n\geq\binom{r_1+d_1}{d_1}\cdots\binom{r_\ell+d_\ell}{d_\ell}-1$. For $i=1,\dots,\ell$, let $X_i$ be a generically $d_i$-parameterized subvariety of $\PP^n$ of dimension $r_i$. Then the Hadamard product $X_1 \hada \cdots \hada X_\ell$ and the product variety $X_1\times \cdots \times X_\ell$ are projectively equivalent as subvarieties of $\PP^n$.
\end{cor}
\begin{proof}
For $i=1,\dots,\ell$, assume that $X_i$ has parametric equations given by
\[
X_i:\begin{cases}
x_0=f_{i0}(y_{i0}, \dots, y_{ir_i})\\
x_1=f_{i1}(y_{i0}, \dots, y_{ir_i})\\
\qquad  \qquad \vdots \\
x_n=f_{in}(y_{i0}, \dots, y_{ir_i})\\
\end{cases}
\]
where $f_{ij}(y_{i0}, \dots, y_{ir_i})\in \mathbb{K}[y_{i0}, \dots, y_{ir_i}]_{d_i}$
, for $j=0, \dots, n$.
\par Since $dim_{\mathbb{K}} \left(\mathbb{K}[y_{i0}, \dots, y_{ir_i}]_{d_i}\right)=\binom{r_i+d_i}{d_i}$, then the linear span of $X_i$ is of dimension $\binom{r_i+d_i}{d_i}-1$.
\par Therefore, by Theorem \ref{teoX1timesXkbign2}, we have that $X_1 \hada \cdots \hada X_\ell$ and $X_1\times \cdots \times X_\ell$ are projectively equivalent as subvarieties of $\PP^n$.
\end{proof}

Corollary \ref{corteoXtimesYgen} easily yields the following Corollary.
\begin{cor}\label{corXtimesYgen}
Let $\ell$ be a positive integer. For $i=1,\dots,\ell$, let $r_i, d_i,$ be positive integers and let $n\geq\binom{r_1+d_1}{d_1}\cdots\binom{r_\ell+d_\ell}{d_\ell}-1$. For $i=1,\dots,\ell$, let $X_i$ be a generically $d_i$-parameterized subvariety of $\PP^n$ of dimension $r_i$. Then:
\begin{enumerate}
  \item[{\it i)}] $dim(X_1\hada\cdots\hada X_\ell)=\sum\limits_{i=1}^{\ell} dim(X_i)$
  \item[{\it ii)}] $deg(X_1\hada\cdots\hada X_\ell)=\binom{r_1+\cdots+r_\ell}{r_1,\dots,r_\ell}\prod\limits_{i=1}^{\ell} deg(X_i)$
  \item[{\it iii)}] $HF_{X_1\hada\cdots\hada X_\ell}=\prod\limits_{i=1}^{\ell} HF_{X_i}$
  \item[{\it iv)}] $X_1\hada\cdots\hada X_\ell$ is non-singular.
\end{enumerate}
\end{cor}

\section{Small ambient space}
\par Recall that, for $X$ and $Y$ subvarieties of $\PP^n$, we set $N=(h+1)(k+1)-1$, where $h$ and $k$ are the dimensions of the linear spans of $X$ and $Y$, respectively. This becomes $N=\binom{r+d_X}{d_X}\binom{s+d_Y}{d_Y}-1$ when $X$ and $Y$ are two generically $d_X$-parameterized and $d_Y$-parameterized subvarieties of $\PP^n$ of dimensions $r$, $s$, respectively.
\par In the previous section, for $n\geq N$, we determined the dimension, the degree and the Hilbert function of the Hadamard product in terms of the same invariants of the factors.
\par Now we consider the range $N-\left(\binom{r+d_X}{d_X}+\binom{s+d_Y}{d_Y}-2\right)\leq n\leq N-1$ in the case of generically $d_X$-parameterized and $d_Y$-parameterized subvarieties. We will see that the dimension and the degree formulas still hold, but the relation on the Hilbert functions fails. Moreover, the Hadamard product can be a singular variety, even if the factors are smooth.
In order to study Hadamard products in a small ambient space we use Segre-Veronese varieties (\cite{CGG}), thus we briefly recall some basic notation about them.
\par Let $\ell$ be a positive integer.
Let $r_1,\dots,r_\ell,d_1,\dots, d_\ell$ be positive integers and set $N=\binom{r_1+d_1}{d_1}\cdots\binom{r_\ell+d_\ell}{d_\ell}-1$. We denote by $S$ the image in $\PP^N$ of a Segre-Veronese embedding of type $(d_1,\dots,d_\ell)$ from $\PP^{r_1}\times\cdots\times\PP^{r_\ell}$ to $\PP^N$.

\begin{teo}\label{teoXhadaYN-1}
Let $r, s, d_X, d_Y$ be positive integers, let $N=\binom{r+d_X}{d_X}\binom{s+d_Y}{d_Y}-1$ and $N-\left(\binom{r+d_X}{d_X}+\binom{s+d_Y}{d_Y}-2\right)\leq n\leq N-1$. Let $X$ and $Y$ be two generically $d_X$-parameterized and $d_Y$-parameterized subvarieties of $\PP^n$ of dimensions $r$, $s$, respectively.
If $n> r+s$, then:
\begin{enumerate}
  \item[{\it i)}] $dim(X\hada Y)=dim(X)+dim(Y)$
  \item[{\it ii)}] $deg(X\hada Y)=\binom{r+s}{s}deg(X)deg(Y)$.
\end{enumerate}
\end{teo}
\begin{proof}
Consider the Segre-Veronese embedding of type $(d_X,d_Y)$ from $\PP^r\times \PP^s$ to $\PP^N$ and let $S$ be its image.
\par Assume that $X$  and $Y$ have parametric equations given respectively  by
\[
X:\begin{cases}
x_0=f_0(y_0, \dots, y_{r})\\
x_1=f_1(y_0, \dots, y_{r})\\
\qquad  \qquad \vdots \\
x_{n}=f_{n}(y_0, \dots, y_{r})\\
\end{cases}
Y:\begin{cases}
x_0=g_0(z_0, \dots, z_{s})\\
x_1=g_1(z_0, \dots, z_{s})\\
\qquad  \qquad \vdots \\
x_{n}=g_{n}(z_0, \dots, z_{s})
\end{cases}
\]
where $f_i(y_0, \dots, y_{r})\in \mathbb{K}[y_0, \dots, y_{r}]_{d_X}$ and $g_i(z_0, \dots, z_{s})\in \mathbb{K}[z_0, \dots, z_{s}]_{d_Y}$, for $i=0, \dots, {n}$.
\par Observe that, for each $i=0, \dots, {n}$, the form $f_ig_i$ has bi-degree $(d_X,d_Y)$ in $\mathbb{K}[y_0, \dots, y_{r},z_0, \dots, z_{s}]$. Since $\mathbb{K}[y_0, \dots, y_{r},z_0, \dots, z_{s}]_{(d_X,d_Y)}$ has dimension $N+1$, then, for each $i=0, \dots, {n}$, $f_ig_i$ defines a point $P_i$ of $(\PP^N)^{^*}$.
\par Since $X$ and $Y$ are generically $d_X$-parameterized and $d_Y$-parameterized subvarieties, the linear span of the points $P_0,\dots,P_{n}$ is of dimension $n$.
\par Consider the $(n+1)\times(N+1)$ matrix $M'$ whose rows are the coordinates of the points $P_0,\dots,P_{n}$. Again since $X$ and $Y$ are generically $d_X$-parameterized and $d_Y$-parameterized subvarieties, $M'$ has maximum rank, hence it defines a projection $\pi$ from $\PP^N$ to $\PP^{n}$ whose center we call $\Lambda$. Note that $ dim(\Lambda)=N-n-1$ and the linear span of the points $P_0,\dots,P_n$ is the dual of $\Lambda$.
\par In order to show the genericity of $\Lambda$ consider the Segre variety $T\subseteq (\PP^N)^{^*}$ defined as the image of the Segre-embedding

{\small{\[
\PP(\mathbb{K}[y_0,\dots,y_r]_{d_X})\times \PP(\mathbb{K}[z_0,\dots,z_s]_{d_Y})\hookrightarrow \PP(\mathbb{K}[y_0,\dots,y_r,z_0,\dots,z_s]_{(d_X,d_Y)}).\]}}

\par Now, any pair of generic parameterizations defines $n+1$ points (the $P_0,\dots,P_n$ above) of $(\PP^N)^{^*}$ belonging to $T$ whose linear span is of dimension $n$. Conversely, any $n+1$ points of $T$ can be obtained from parameterizations (with suitable coefficients) of two subvarieties of $\PP^{n}$ with parametric representation (of the given dimensions and degrees).
\par On the other hand, for any generic linear subspace $L$ of $(\PP^N)^{^*}$ of dimension $n$, defined by $N-n$ generic hyperplanes $H_1,\dots,H_{N-n}$, we shall consider $T_i=T\cap H_1\cap \cdots \cap H_i$. Since $n\geq N-dim(T)=N-\left(\binom{r+d_X}{d_X}+\binom{s+d_Y}{d_Y}-2\right)$, we have that $dim(T_i)\geq 2$ for all $i=1,\dots,N-n-2$ and $dim(T_{N-n-1})\geq 1$. Therefore by \cite[Proposition 18.10]{Harris}, $T_{N-n}$ contains at least $n+1$ points which generate $L$. Thus we may assume that the linear subspaces of $(\PP^N)^{^*}$ of dimension $n$ generated by $n+1$ points of $T$ are generic, and so $\Lambda$ is generic as well.

\par For  $n\geq r+s=dim(S)$, since $\Lambda$ is generic,  we have  $dim(\pi(S))=dim(S)=r+s$. Since $n>r+s$, we also have $\pi(S)\neq\PP^n$, and so the projection $\pi_{\vert_S}:S\to\pi(S)$ is a birational map. Hence $deg(\pi(S))=deg(S)=\binom{r+s}{s}d_Xd_Y$.
\par Set $\Sigma=\{P\hada Q|P\in X,Q\in Y\}$. It is easy to see that $\pi(S)\subseteq \Sigma \subseteq X\hada Y$. Since $r+s=dim (\pi(S))\leq dim (X\hada Y)\leq r+s,$ we have that $\pi(S)=X\hada Y$, and so $dim(X\hada Y)=dim(\pi(S))=r+s$ and $deg(X\hada Y)=deg(\pi(S))=\binom{r+s}{s}d_Xd_Y$.

\end{proof}

\begin{oss}{\rm
As in Remark \ref{closure}, under the assumptions of Theorem \ref{teoXhadaYN-1}, we also proved that $X\hada Y$ which is, by definition, $\overline{\{P\hada Q|P\in X,Q\in Y\}}$ turns out to be $\{P\hada Q|P\in X,Q\in Y\}.$
}
\end{oss}
\vskip 0.5cm
In order to make Theorem \ref{teoXhadaYN-1} more effective, we can find  explicit numerical conditions on $X$ and $Y$ so that $n\geq N-\left(\binom{r+d_X}{d_X}+\binom{s+d_Y}{d_Y}-2\right)$  yields $n> r+s$.
\begin{lemm}
	Using the notations of Theorem \ref{teoXhadaYN-1}, we have that:
if $(d_X,d_Y,r,s)$ is in the following table, then $N-\left(\binom{r+d_X}{d_X}+\binom{s+d_Y}{d_Y}-2\right)> r+s$.
		\begin{center}
			\begin{tabular}{|c|c|c|c|}
				\hline
				$d_X$&$d_Y$ &$r$ &$s$ \\ \hline
				$\geq 3$ &$\forall$ &$\forall$ &$\forall$  \\  \hline
				$\forall$ & $\geq 3$ & $\forall$ &$\forall$  \\  \hline
				$2$ &$\geq 2$ &$\forall$ &$\forall$  \\  \hline
				$2$ & $1$  &$\forall$ & $\geq 2$ \\  \hline
				$2$ &$1$ &$\geq 2$ &$1$  \\  \hline
				$\geq 2$ &$2$ &$\forall$ &$\forall$  \\  \hline
				$1$ & $2$ & $\geq 2$ &$\forall$  \\  \hline
				$1$ &$2$ &$1$ &$\geq 2$  \\  \hline
				$1$ &$1$ &$\geq 3$ & $\geq 2$ \\  \hline
				$1$ &$1$ &$\geq 2$ & $\geq 3$ \\  \hline
			\end{tabular}
		\end{center}
\end{lemm}

\vskip 0.5cm
\begin{oss}\label{rmkHF}{\rm
Notice that, in the hypotheses of Theorem \ref{teoXhadaYN-1}, we have $HF_{X\hada Y}\not=HF_XHF_Y$. In fact, since $X$ is not contained in a linear subspace of dimension less than $\binom{r+d_X}{d_X}-1$ and similarly $Y$, we have
$$
HF_X(1)=HF{_{\PP^{\binom{r+d_X}{d_X}-1}}}(1)=\binom{r+d_X}{d_X}
$$
and
$$
HF_Y(1)=HF{_{\PP^{\binom{s+d_Y}{d_Y}-1}}}(1)=\binom{s+d_Y}{d_Y}
$$
and so
$$
HF_X(1)HF_Y(1)=\binom{r+d_X}{d_X}\binom{s+d_Y}{d_Y}> N \geq HF_{X\hada Y}(1).
$$}
\end{oss}
\begin{oss}{\rm
In Remark \ref{rmkM'} we saw that $M'$ being of maximum rank is sufficient to have the formulas for the dimension, the degree and the Hilbert function, when $n\geq N$. When $n<N$, besides the failure of the Hilbert function formula (Remark \ref{rmkHF}), $M'$ of maximum rank does not grant the degree formula, as Example \ref{Ex3.6} shows.
}
\end{oss}
Using a similar technique to that contained in the proof of Theorem \ref{teoXhadaYN-1}, we can extend such Theorem to a finite number of subvarieties.

\begin{teo}\label{teoX1hadaXkN-1}
Let $\ell$ be a positive integer. For $i=1,\dots,\ell$, let $r_i, d_i,$ be positive integers, let $N=\binom{r_1+d_1}{d_1}\cdots\binom{r_\ell+d_\ell}{d_\ell}-1$ and $N-\left(\binom{r_1+d_1}{d_1}+\cdots+\binom{r_\ell+d_\ell}{d_\ell}-\ell\right)\leq n\leq N-1$. For $i=1,\dots,\ell$, let $X_i$ be a generically $d_i$-parameterized subvariety of $\PP^{n}$ of dimension $r_i$.
If $n> r_1+\cdots+r_\ell$, then:
\begin{enumerate}
  \item[{\it i)}] $dim(X_1\hada\cdots\hada X_\ell)=\sum\limits_{i=1}^{\ell} dim(X_i)$
  \item[{\it ii)}] $deg(X_1\hada\cdots\hada X_\ell)=\binom{r_1+\cdots+r_\ell}{r_1,\dots,r_\ell}\prod\limits_{i=1}^{\ell} deg(X_i)$.
\end{enumerate}
\end{teo}
Now we provide a numerical condition for the Hadamard product to be smooth and we give an estimate on how big the singular locus is when singularities occur. In order to do this we will use the variety of secant lines to a subvariety $\mathcal{S}$ that we denote by $\sigma_2(\mathcal{S})$. It is nothing but the closure of the union of the lines joining two distinct points of $\mathcal{S}$.
\par Notice that, for $n$ in our range, when using generically $d$-parameterized subvarieties of $\PP^n$, we are dealing with smooth varieties, as the following Proposition shows.
\begin{prop}\label{Lemmasmooth}
Let $r, s, d_X, d_Y$ be positive integers, let $N=\binom{r+d_X}{d_X}\binom{s+d_Y}{d_Y}-1$ and $N-\left(\binom{r+d_X}{d_X}+\binom{s+d_Y}{d_Y}-2\right) \leq n\leq N-1$. Let $X$ and $Y$ be two generically $d_X$-parameterized and $d_Y$-parameterized subvarieties of $\PP^n$ of dimensions $r$, $s$, respectively. Then $X$ and $Y$ are non-singular.
\end{prop}
\begin{proof}
We only prove that $X$ is non-singular (similarly for $Y$).
\par By Remark \ref{rmkgennonsingular} $X$ is non-singular for $n\geq \binom{r+d_X}{d_X}-1$ and we will prove that this is always the case.
To this end, observe that $\binom{s+d_Y}{d_Y}\geq 2$ and so
\[
\binom{s+d_Y}{d_Y}\left(\binom{r+d_X}{d_X}-1\right)\geq 2 \left(\binom{r+d_X}{d_X}-1\right),
\]
thus
\[
n\geq N-\left(\binom{r+d_X}{d_X}+\binom{s+d_Y}{d_Y}-2\right) =
\]
\[
 \binom{r+d_X}{d_X}\binom{s+d_Y}{d_Y}-1 - \left(\binom{r+d_X}{d_X}+\binom{s+d_Y}{d_Y}-2\right) \geq \binom{r+d_X}{d_X}-1.
\]
\end{proof}

Now we want to see when the Hadamard product of generically $d$-parameterized subvarieties is non-singular and how big the singular locus is when singularities show up. We start with a more general statement in the line of \cite{R1,R2}, which will apply to our case.

\begin{teo}\label{teosing}
Let $\mathcal{S}\subseteq\PP^m$ be a smooth irreducible subvariety, $n<m$ and $\mathcal{S}'\subseteq\PP^n$ the image of $\mathcal{S}$ under a generic projection.
\begin{enumerate}
  \item[{\it i)}]  If $n\geq dim(\sigma_2(\mathcal{S}))$, then $\mathcal{S}'$ is smooth.
  \item[{\it ii)}] If $dim(\mathcal{S})<n<dim(\sigma_2(\mathcal{S}))$, then  $ dim(\mbox{Sing}(\mathcal{S}'))\geq 2dim(\mathcal{S})-n$.
\end{enumerate}
\end{teo}
\begin{proof}
Let $\pi$ be the projection from $\PP^m$ to $\PP^n$ whose center is a generic linear subspace $\Lambda$ of dimension $m-n-1$ and let $\sigma_2=\sigma_2(\mathcal{S}).$
\\ \noindent {\it i)} If $n\geq dim(\sigma_2)$, since $\Lambda$ is generic, we have that $\Lambda\cap \sigma_2=\emptyset$, then $\mathcal{S}'$ is smooth.
\\ \noindent {\it ii)}  Define the incidence correspondence $\Theta \subseteq \mbox{Sing}(\mathcal{S}')\times(\Lambda\cap\sigma_2)$ where \[\Theta=\lbrace(Q,P): \{Q\}=\pi(r_P\setminus\Lambda), r_P \mbox{ is a tangent or secant line to $\mathcal{S}$ through } P \rbrace.\]
We consider the projection maps $p_1:\Theta\rightarrow\mbox{Sing}(\mathcal{S}')$ and $p_2:\Theta\rightarrow\Lambda\cap\sigma_2$.
First we prove that $p_1$ has a finite fiber over a point $Q\in\mbox{Sing}(\mathcal{S}')$. Since $\Lambda$ is a hyperplane in $\overline{\pi^{-1}(Q)}$, and $\Lambda\cap \mathcal{S}=\emptyset$, then $\overline{\pi^{-1}(Q)}\cap \mathcal{S}$ contains only a finite number of points. Since each secant, or tangent, line to $\mathcal{S}$ contains points of $\mathcal{S}$, then $\overline{\pi^{-1}(Q)}$ contains a finite number of secant, or tangent, lines to $\mathcal{S}$; by the genericity of $\Lambda$ each of these lines contains a finite number of points of $\Lambda\cap\sigma_2$. Hence, $p_1^{-1}(Q)$ is finite. Now we consider the generic fiber of $p_2$ over $P\in\Lambda\cap\sigma_2$. Since the family of secant and tangent lines to $\mathcal{S}$ through $P$ has dimension at least $2dim(\mathcal{S})+1- dim(\sigma_2),$ then so does the generic fiber of $p_2$. Since $ dim(\Lambda\cap\sigma_2)= dim(\Lambda)+ dim(\sigma_2)-m$, we conclude that
$$dim(\mbox{Sing}(\mathcal{S}'))= dim(\Theta)\geq  dim(\Lambda)+2dim(\mathcal{S})+1-m=2dim(\mathcal{S})-n.$$
\end{proof}

\begin{cor}\label{propsing}
Let $r, s, d_X, d_Y$ be positive integers, let $N=\binom{r+d_X}{d_X}\binom{s+d_Y}{d_Y}-1$ and $N-\left(\binom{r+d_X}{d_X}+\binom{s+d_Y}{d_Y}-2\right)\leq n\leq N-1$. Let $X$ and $Y$ be two generically $d_X$-parameterized and $d_Y$-parameterized subvarieties of $\PP^n$ of dimensions $r$, $s$, respectively. Let $S$ be the Segre-Veronese embedding of type $(d_X,d_Y)$ of $\mathbb{P}^r\times\mathbb{P}^s$.
\begin{enumerate}
  \item[{\it i)}]  If $n\geq dim(\sigma_2(S))$, then $X\hada Y$ is smooth.
  \item[{\it ii)}] If $r+s<n<dim(\sigma_2(S))$, then  $ dim(\mbox{Sing}(X\hada Y))\geq 2r+2s-n$.
\end{enumerate}
\end{cor}
\begin{proof}
  Since the projection in the proof of Theorem \ref{teoXhadaYN-1} is generic, we may replace $m$ with $N$ and $\mathcal{S}$ with $S$ in Theorem \ref{teosing}, so that $\mathcal{S'}=X\hada Y.$
\end{proof}

\begin{oss}{\rm
If $X$ and $Y$ are not generic enough, it can happen that $dim(Sing(X\hada Y))$ is smaller than $2r+2s-n$, as Example \ref{Ex3.6} shows.
\par Also note that the bound of Corollary \ref{propsing}-{\it ii)} can be sharp, as Example \ref{Ex3.1} shows.
}
\end{oss}
\begin{oss}{\rm
If $(d_X,d_Y)=(1,1)$, then $\sigma_2(S)$ can be identified with the variety of $r\times s$ matrices of rank at most $2$ and so $dim(\sigma_2(S))=2r+2s-1$.
\par If $(d_X,d_Y)\neq(1,1)$, by \cite[Theorem 4.2]{AB}, we have that $$dim(\sigma_2(S))=\min\{N,2r+2s+1\},$$ and it is easy to check that $dim (\sigma_2(S))=2r+2s+1$.
}
\end{oss}
\begin{oss}{\rm
In the case $(d_X,d_Y)=(1,1)$, Corollary \ref{propsing} yields that $X\hada Y$ is either smooth or $dim(\mbox{Sing}(X\hada Y))\geq 2r+2s-n>2r+2s-dim(\sigma_2(S))=1.$
Thus, if $X\hada Y$ is not smooth, it is singular at least along a surface.
}
\end{oss}

The following conditions show that the hypotheses of Corollary \ref{propsing} hold in a large number of cases.
\begin{lemm}\label{Lemmatabelle}
Using the notations of Corollary \ref{propsing}, we have that:
\begin{itemize}
\item[{\it i)}]  If either $d_X\geq 6$ or $d_Y\geq 6$, then $N-\left(\binom{r+d_X}{d_X}+\binom{s+d_Y}{d_Y}-2\right)\geq dim(\sigma_2(S))$.

\item[{\it ii)}]  If  $(d_X,d_Y,r,s,n)$ is in the following table, then $N-\left(\binom{r+d_X}{d_X}+\binom{s+d_Y}{d_Y}-2\right)\leq n\leq N-1$ and $r+s<n< dim(\sigma_2(S))$.
\begin{center}
 \begin{tabular}{|c|c|c|c|c|}
 \hline
 $d_X$&$d_Y$ &$r$ &$s$ &$n$\\ \hline
 $2$ &$2$ &$1$ &$1$ & $n=4$ \\  \hline
 $4$ &$1$ &$1$ &$1$ & $n=4$ \\  \hline
 $3$ &$1$ &$1$ &$\leq 3$ & $3s\leq n\leq 2s+2$ \\  \hline
 $2$ &$1$ &$1$ &$\forall$ & $2s\leq n\leq 2s+2$ \\  \hline
 $2$ &$1$ &$2$ &$1$ & $5\leq n\leq 6$ \\  \hline
 $1$ &$4$ &$1$ &$1$ & $n=4$ \\  \hline
 $1$ &$3$ &$\leq 3$ &$1$ & $3r\leq n\leq 2r+2$ \\  \hline
 $1$ &$2$ &$\forall$ &$1$ & $2r+1\leq n\leq 2r+2$ \\  \hline
 $1$ &$2$ &$1$ &$2$ & $5\leq n\leq 6$ \\  \hline
 $1$ &$1$ &$1$ &$\forall$  & $s+2\leq n\leq 2s$\\  \hline
 $1$ &$1$ &$2$ &$\geq 3$  & $2s\leq n\leq 2s+2$\\  \hline
 $1$ &$1$ &$2$ &$2$  & $5\leq n\leq 6$\\  \hline
 $1$ &$1$ &$2$ &$1$  & $n=4$\\  \hline
 $1$ &$1$ &$3$ & $\leq 5$ & $3s\leq n\leq 2s+4$\\  \hline
 $1$ &$1$ &$\forall$ & $1$ & $r+2\leq n \leq 2r$ \\ \hline
 $1$ &$1$ &$\geq 3$ & $2$ & $2r\leq n\leq 2r+2$ \\ \hline
  $1$ &$1$ &$2$ & $2$ & $5\leq n\leq 6$ \\ \hline
   $1$ &$1$ &$1$ & $2$ & $n=4$ \\ \hline
 $1$ &$1$ &$\leq 5$ & $3$ & $3r\leq n\leq 2r+4$\\ \hline
\end{tabular}
\end{center}

	\end{itemize}
\end{lemm}
\begin{oss}{\rm
In the cases of Lemma \ref{Lemmatabelle}-{\it i)}, for small values of $d_X$ and $d_Y$, the cases in which the inequality  $N-\left(\binom{r+d_X}{d_X}+\binom{s+d_Y}{d_Y}-2\right)\geq dim(\sigma_2(S))$ holds can be determined in terms of $r$ and $s$.
}
\end{oss}

\begin{oss}\label{remsegremaggell}{\rm
Let $S$ be the Segre-Veronese variety with $\ell>2$. By \cite[Theorem 4.2]{AB}, $S$ does not have a defective secant variety, and thus $$dim(\sigma_2(S))=\min\left\{N,2\left( r_1+\cdots+r_\ell\right)+1\right\},$$
and it is easy to check that $dim(\sigma_2(S))=2\left( r_1+\cdots+r_\ell\right)+1$. 	
}
\end{oss}
\vskip 0.5cm
Notice that Proposition \ref{Lemmasmooth} easily extends to a finite number of varieties.
Moreover by using Remark \ref{remsegremaggell}, Corollary \ref{propsing} can be extended to a finite number of varieties.
\begin{prop}\label{propX1hadaXkN-1}
Let $\ell>2$. For $i=1,\dots,\ell$, let $r_i, d_i,$ be positive integers, let $N=\binom{r_1+d_1}{d_1}\cdots\binom{r_\ell+d_\ell}{d_\ell}-1$ and $N-\left(\binom{r_1+d_1}{d_1}+\cdots+\binom{r_\ell+d_\ell}{d_\ell}-\ell\right)\leq n\leq N-1$. For $i=1,\dots,\ell$, let $X_i$ be a generically $d_i$-parameterized subvariety of $\PP^{n}$ of dimension $r_i$.
\begin{enumerate}
  \item[{\it i)}] If $n\geq 2\left( r_1+\cdots+r_\ell\right)+1$, then $X_1\hada\cdots\hada X_\ell$ is smooth;
  \item[{\it ii)}] if $\left( r_1+\cdots+r_\ell\right)< n< 2\left( r_1+\cdots+r_\ell\right)+1$, then   $ dim(\mbox{Sing}(X_1\hada\cdots\hada X_\ell))\geq 2\left( r_1+\cdots+r_\ell\right)-n$.
\end{enumerate}
\end{prop}

\section{Some examples}\label{section4}
Here we collect some examples to show the role of the genericity assumption in our results; we use \verb|CoCoA| (\cite{CoCoA}), following the procedure given in \cite{BCFL1}.

In Example \ref{Ex2.7} we have $n\geq N$, but $X$ and $Y$ are not generic enough to have the matrix $M'$ of maximum rank (see Theorem \ref{teoX1timesXkbign2} and Corollary \ref{corX1timesXkbign}). Also, the varieties $X$ and $Y$ are both non singular, but $\mbox{Sing}(X\hada Y)\neq \emptyset$, and so $X\hada Y$ is neither projectively equivalent nor isomorphic to the product variety $X\times Y$.

In Example \ref{Ex3.6} we have $n< N$, $X$ and $Y$ are generic enough to have the matrix $M'$ of maximum rank, but, $X$ and $Y$ are not generic enough to give a generic center of projection $\Lambda$ (see Theorem \ref{teoXhadaYN-1} and Corollary \ref{propsing}).
Also, the degree formula and the lower bound on the dimension of the singular locus do not hold.
\par In Example \ref{Ex3.1} the dimension of the singular locus is equal to the lower bound.

Finally we give an example (Example \ref{Exnocomp}) which is not computable but can be directly deduced from our results.

\begin{ex}\label{Ex2.7}{\rm
		Let $X$ be the line of $\PP^5$ given by the equations $\{x_0-x_1=0,x_0-x_2=0,x_3-x_5=0,x_0+x_3-x_4=0\}$ and let $Y$ be the conic of $\PP^5$ given by the equations $\{x_0-2x_3+3x_5=0,x_1+x_4-x_5=0,x_2+2x_3-3x_4=0,x_0^2+x_1^2+x_2^2+x_3^2+x_4^2+x_5^2+5x_0x_1+8x_0x_1-2x_2x_5+10x_0x_4=0\}$. Here $h=1$ and $k=2$ and so $N=(h+1)(k+1)-1=5$.

		Computations show that the Hadamard product has dimension $2=r+s=dim(X)+dim(Y)$ and degree $4=\binom{r+s}{r}deg(X)deg(Y)$ as expected, but $HF_{X\hada Y}\neq HF_X HF_Y$. Also, the singular locus has dimension $0$ and degree $5$.
		\par In this case the matrix $M'$ does not have maximum rank. In fact, first we write the parameterizations of $L_1=X$ and of the plane $L_2$ containing $Y$:
		\[
		L_1:\begin{cases}
		x_0=y_1\\
		x_1=y_1\\
		x_2=y_1\\
		x_3=y_0\\
		x_4=y_0+y_1\\
		x_5=y_0\\
		\end{cases}
		L_2:\begin{cases}
		x_0=2z_0-3z_2\\
		x_1=-z_1+z_2\\
		x_2=-2z_0+3z_1\\
		x_3=z_0\\
		x_4=z_1\\
		x_5=z_2\\
		\end{cases}
		\]
		and then we obtain
		\[
		M'
		=\begin{pmatrix}
		0 & 0 & 0 & 2 & 0 & -3\\
		0 & 0 & 0 & 0 & -1 & 1\\
		0 & 0 & 0 & -2 & 3 & 0\\
		1 & 0 & 0 & 0 & 0 & 0\\
		0 & 1 & 0 & 0 & 1 & 0\\
		0 & 0 & 1 & 0 & 0 & 0\\
		\end{pmatrix}
		\]
		whose determinant equals $0$.
	}
\end{ex}

\begin{ex}\label{Ex3.6}{\rm
Let $X$ be the line of $\PP^4$ given by the equations $\{x_0-x_1=0,x_0-x_2=0,x_3-2x_4=0\}$ and let $Y$ be the conic in $\PP^4$ given by the equations $\{x_0-x_3=0,x_1-x_4=0,x_1^2-x_0x_2=0\}$.

Computations show that $X\hada Y$ has dimension $2=r+s=dim(X)+dim(Y)$ but it has degree $3<\binom{r+s}{s}dim(X)dim(Y)$.
\par Surprisingly enough $X\hada Y$ does not have singularities and $dim(Sing(X\hada Y))<2r+2s-n=0$ (see Corollary \ref{propsing}). Moreover $M'$ has maximum rank. In fact, writing the parameterization of $X$ and $Y$
\[
X:\begin{cases}
x_0=y_0-y_1\\
x_1=y_0-y_1\\
x_2=y_0-y_1\\
x_3=y_0\\
x_4=2y_0\\
\end{cases}
\ \ \ \ \ Y:\begin{cases}
x_0=z_0^2\\
x_1=z_0z_1\\
x_2=z_1^2\\
x_3=z_0^2\\
x_4=z_0z_1\\
\end{cases}
\]
we obtain
$$
M'=\left(
      \begin{array}{cccccc}
        -1 & 0 & 0 & -1 & 0 & 0 \\
        0 & 1 & 0 & 0 & -1 & 0 \\
        0 & 0 & 1 & 0 & 0 & 1\\
        1 & 0 & 0 & 0 & 0 & 0 \\
        0 & 2 & 0 & 0 & 0 & 0 \\
      \end{array}
    \right).
$$
\par Note that $\Lambda$ is the point $[0:0:-2:0:0:1]$ and so it belongs to the Segre-Veronese variety $S$ and this is why our genericity hypothesis on $X$ and $Y$ is not satisfied.
}
\end{ex}
\begin{ex}\label{Ex3.1}{\rm
		Let $X$ be the line of $\PP^3$ given by the equations $\{x_0+x_1+x_2+2x_3=x_0-x_1+4x_2-x_3=0\}$ and let $Y$ be the conic of $\PP^3$ given by the equations $\{x_0+2x_1+3x_2+x_3=x_0^2+2x_0x_2+2x_0x_3+x_1^2+2x_1x_2-2x_1x_3+x_2^2+2x_2x_3+x_3^2=0\}$.
Here $r=s=1$, $d_X=1$ and $d_Y=2$, so $3$ is the minimum possible value for $n$, moreover we are in the case {\it ii)} of  Corollary \ref{propsing}.
\par In this case $X\hada Y$ is a singular quartic surface and the singular locus is {\em exactly} of dimension $1=2r+2s-n$.}
\end{ex}

\begin{ex}\label{Exnocomp}{\rm
Let $k$ be a positive integer. Let $\mathcal{C}$ be a generic  plane conic in $\PP^{2k+1}$.
Let $L $ be a generic linear subspace of $\PP^{2k+1}$ of dimension $k$. In view of
Lemma \ref{Lemmatabelle}, we can use Theorem \ref{teoXhadaYN-1} and Corollary
\ref{propsing} to obtain $dim(\mathcal{C}\hada  L)=k+1$,
$deg(\mathcal{C}\hada L)=\binom{k+1}{k}\cdot 2\cdot 1=2(k+1)$ and
$dim(Sing(\mathcal{C}\hada  L))\geq 2+2k-(2k+1)=1$.}
\end{ex}

\end{document}